\documentclass{amsart}
\usepackage{amsmath}
\usepackage{amssymb}
\usepackage{amscd}
\usepackage[all]{xy}
\usepackage{color}
\usepackage{tikz-cd}
\usepackage{ulem}
\usepackage{tikz}
\usepackage{comment}
\usepackage{graphicx}

\mathchardef\mhyphen="2D

\newtheorem{theorem}{Theorem}[section]
\newtheorem{lemma}[theorem]{Lemma}
\newtheorem{corollary}[theorem]{Corollary}
\newtheorem{proposition}[theorem]{Proposition}

\newtheorem{example}[theorem]{Example}
\newtheorem{remark}[theorem]{Remark}

\begin{document}
\title[Local and $t$-local properties of integral domains]
{Local properties of integral domains under extensions and pullback constructions}

\author [H. Baek] {Hyungtae Baek}
\address{(Baek) School of Mathematics,
Kyungpook National University, Daegu 41566,
Republic of Korea}
\email{htbaek5@gmail.com}

\author [J. W. Lim] {Jung Wook Lim}
\address{(Lim) Department of Mathematics,
College of Natural Sciences,
Kyungpook National University, Daegu 41566,
Republic of Korea}
\email{jwlim@knu.ac.kr}

\author[O. Ouzzaouit]{Omar Ouzzaouit}
\address{(Ouzzaouit) Laboratory of Research in Sciences and Technique (LRST), The Higher School of Education and Training, Ibnou Zohr University, Agadir, Morocco}
\email{o.ouzzaouit@uiz.ac.ma}

\author[A. Tamoussit]{Ali Tamoussit}
\address{(Tamoussit) Department of Mathematics, The Regional Center for Education and Training Professions Souss Massa, Inezgane, Morocco}
\email{a.tamoussit@crmefsm.ac.ma; tamoussit2009@gmail.com}

\thanks{Words and phrases: Local property, ring extension, pullback construction, Krull-like domain}

\thanks{$2020$ Mathematics Subject Classification: 13A15, 13B30}

\date{\today}

\begin{abstract}
For a property $\mathcal{X}$ of integral domains,
an integral domain $D$ is said to be a {\it locally $\mathcal{X}$-domain}
if $D_P$ has the property $\mathcal{X}$ for every prime ideal $P$ of $D$.
In this paper, we study the transfer of local properties of integral domains
under several extensions and constructions, including flat overrings,
Nagata ideal transforms, polynomial rings and their quotient extensions, and pullback constructions.
\end{abstract}

\maketitle

\section{Introduction}

\subsection{Local properties}

The localization plays a fundamental role in commutative algebra.
In particular, it provides a useful technique
in studying some classes of commutative rings and modules.
In fact, the localization at a prime ideal allows one to restrict all studies to the local case.
Moreover, many classes of integral domains are defined or characterized
in terms of their localizations on distinguished types of ideals.
For example, Pr\"ufer domains, almost Dedekind domains \cite{Gi64}
and almost Krull domains \cite{P68} are defined by localization at maximal ideals,
$t$-almost Dedekind domains \cite{K89} (respectively, P$v$MDs) are defined
(respectively, characterized)
by localization at maximal $t$-ideals, and $P$-domains \cite{MZ81}
({\it i.e.}, locally essential domains) are defined by localization at associated prime ideals.
For a property $\mathcal{X}$ of integral domains,
we say that $D$ is a {\it locally $\mathcal{X}$-domain}
(respectively, {\it $t$-locally $\mathcal{X}$-domain})
if any localization of $D$ at a prime ideal (respectively, prime $t$-ideal)
has the property $\mathcal{X}$.
Under this naming system,
an almost Dedekind domain (respectively, almost Krull domain, P$v$MD, and $t$-almost Dedekind domain)
is exactly a locally DVR (respectively, locally Krull domain, $t$-locally valuation domain, and $t$-locally DVR).

\subsection{Star-operations}

To help readers better understand this paper,
we review some definitions and notation related to star-operations.
Let $D$ be an integral domain with quotient field $K$.
Let ${\bf F}(D)$ be the set of nonzero fractional ideals of $D$.
For an $I \in {\bf F}(D)$,
set $I^{-1} := \{a \in K \,|\, aI \subseteq D\}$.
The mapping on ${\bf F}(D)$ defined by $I \mapsto I_v := (I^{-1})^{-1}$
is called the {\it $v$-operation} on $D$;
the mapping on ${\bf F}(D)$ defined by
$I \mapsto I_t :=
\bigcup \{J_v \,|\, J \text{ is a nonzero finitely generated fractional subideal of $I$}\}$
is called the {\it $t$-operation} on $D$.
An ideal $J$ of $D$ is a {\it Glaz–Vasconcelos ideal}
(for short a {\it GV-ideal}),
and denoted by $J \in {\rm GV}(D)$ if
$J$ is finitely generated and $J_v = D$.
For each $I \in {\bf F}(D)$,
the {\it $w$-envelope} of $I$ is the set
$I_w := \{ x \in K \,|\, xJ \subseteq I \text{ for some } J \in {\rm GV}(D)\}$.
The mapping on ${\bf F}(D)$ defined by
$I \mapsto I_w$ is called a {\it $w$-operation} on $D$.
Let $*= v,t$ or $w$.
A nonzero fractional ideal $F$ of $D$ is a {\it fractional $*$-ideal} if $F_* = F$.
A nonzero ideal of $D$ is a {\it prime $*$-ideal} if it is both a $*$-ideal and a prime ideal.
A nonzero proper ideal of $D$ is called a {\it maximal $*$-ideal} if it is maximal
among the proper $*$-ideals of $D$, and denoted by $*\mhyphen{\rm Max}(D)$.
It is worth noting that any maximal $*$-ideal (if it exists) is a prime ideal.

\subsection{Krull-like domains and related integral domains}

Let $D$ be an integral domain, $X^1(D)$ the set of height-one prime ideals of $D$ and
$\mathfrak{A}$ a set of prime ideals of $D$.
Consider the following five conditions: 
\begin{enumerate}
\item[\rm(i)] $D = \bigcap_{P \in \mathfrak{A}} D_P$,
\item[\rm(ii)] $D_P$ is a Noetherian domain for any $P \in \mathfrak{A}$,
\item[\rm(iii)] $D_P$ is a valuation domain for any $P \in \mathfrak{A}$,
\item[\rm(iv)] $D_P$ is a DVR for any $P \in \mathfrak{A}$, and
\item[\rm(v)] every nonzero element of $D$ belongs to only finitely many elements of $\mathfrak{A}$.
\end{enumerate}
Based on these statements, we define the following:
\begin{itemize}
\item[(1)]
$D$ is a {\it Krull domain} if $D$ satisfies (i), (iv) and (v) when $\mathfrak{A} = X^1(D)$,
\item[(2)]
$D$ is an {\it infra-Krull domain} if $D$ satisfies (i), (ii) and (v) when $\mathfrak{A} = X^1(D)$,
\item[(3)]
$D$ is a {\it generalized Krull domain} (in the sense of Gilmer \cite{Gi72})
if $D$ satisfies (i), (iii) and (v) when $\mathfrak{A} = X^1(D)$,
\item[(4)]
$D$ is a {\it weakly Krull domain} if $D$ satisfies (i) and (v) when $\mathfrak{A} = X^1(D)$,
\item[(5)]
$D$ is an {\it essential domain} if there exists $\mathfrak{A}$ which makes $D$ satisfy (i) and (iii),
\item[(6)]
$D$ is a {\it Krull-type domain} if there exists $\mathfrak{A}$ which makes $D$ satisfy (i), (iii) and (v),
\item[(7)]
$D$ is a {\it Pr\"{u}fer $v$-multiplication domain} (for short, P$v$MD)
if $D$ satisfies (iii) when $\mathfrak{A} = t\mhyphen{\rm Max}(D)$, and
\item[(8)]
$D$ is a {\it strong Mori domain} (for short, SM domain) if
$D$ satisfies (ii) and (v) when $\mathfrak{A} = t\mhyphen{\rm Max}(D)$.
\end{itemize}

The following are well-known facts about Krull-like domains:
a weakly Krull domain is exactly a one $t$-dimensional domain with finite $t$-character \cite[Lemma 2.1(1)]{AMZ 1992};
an infra-Krull domain is exactly a one $t$-dimensional strong Mori domain \cite[Page 199]{Kim 2011}; and
a generalized Krull domain is precisely a one $t$-dimensional P$v$MD with finite $t$-character \cite[Theorems 3 and 5]{GT18}
(recall that $D$ is a {\it one $t$-dimensional domain} if $t\mhyphen{\rm Max}(D) = X^1(D)$, and
$D$ has {\it finite $t$-character} if every nonzero-nonunit element belongs to only finitely many maximal $t$-ideals of $D$).
Hence we can easily obtain the following implications.

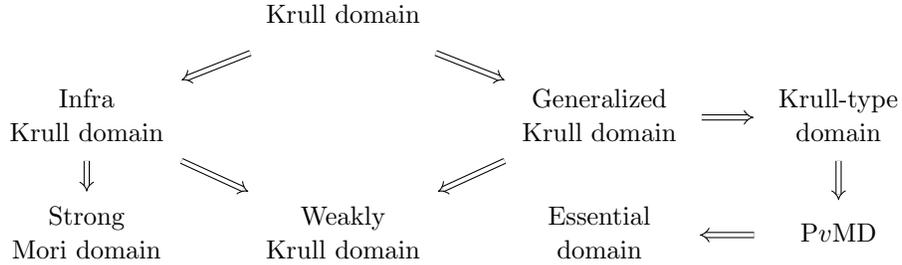
\begin{center}
$
\begin{tikzcd}[arrows=Rightarrow, row sep=1.2em, column sep=2em]
&{\begin{array}{c}\text{Krull domain}\arrow[dr, shorten < = 10pt, shorten > = -3pt]\arrow[dl, shorten < = 10pt, shorten > = -3pt]\end{array}}&&\\
{\begin{array}{c}\text{Infra}\\\text{Krull domain}\end{array}}\arrow[d]\arrow[dr]
&
&{\begin{array}{c}\text{Generalized}\\\text{Krull domain}\end{array}}\arrow[r]\arrow[dl]&{\begin{array}{c}\text{Krull-type}\\ \text{domain}\end{array}}\arrow[d, shorten >= 5pt]\\
{\begin{array}{c}\text{Strong} \\ \text{Mori domain}\end{array}} &{\begin{array}{c}\text{Weakly}\\\text{Krull domain}\end{array}}&{\begin{array}{c}\text{Essential}\\ \text{domain}\end{array}}&\text{P$v$MD}\arrow[l, shorten >= 9pt, shorten <= 13pt]
\end{tikzcd}
$
\end{center}

\subsection{Main results}

In this paper, we study the transfer of local properties of integral domains
to several types of extensions, such as flat overrings,
Nagata ideal transforms, and polynomial ring extensions,
as well as to pullback constructions.
More precisely, for a given property $\mathcal{X}$ of integral domains,
we show that the property of being a locally $\mathcal{X}$-domain
(respectively, $t$-locally $\mathcal{X}$-domain) is preserved under flat overring
(respectively, $t$-flat overring) extensions (Proposition \ref{cffalt}).

Moreover, for non-quasi-local domains, we characterize locally $\mathcal{X}$-domains
in terms of their Nagata ideal transforms (Proposition \ref{NT}).
We further show that, under suitable assumptions on $\mathcal{X}$,
a polynomial ring, its base ring, and certain quotient extensions
share the same locally $\mathcal{X}$ property
(Theorems~\ref{local property QE} and~\ref{t-local property QE}).
Finally, we investigate necessary and sufficient conditions
under which a pullback construction yields a locally Krull-like domain
(Theorems~\ref{locally Krull-like domain} and~\ref{GK domain}).

\section{Some extension of integral domains and their local properties}\label{section 2}

In this section, we investigate the transfer of some local properties
in various extensions of integral domains.

Let $D$ be an integral domain.
Given a property $\mathcal{X}$ of integral domains,
$D$ is said to be a {\it locally $\mathcal{X}$-domain}
(respectively, {\it $t$-locally $\mathcal{X}$-domain})
if $D_P$ has the property $\mathcal{X}$ for any $P \in {\rm Spec}(D)$
(respectively, $P \in t \mhyphen {\rm Spec}(D)$).
Clearly, every locally $\mathcal{X}$-domain is a $t$-locally $\mathcal{X}$-domain,
and if $\mathcal{Y}$ is a property of integral domains such that
$\mathcal{X}$ is a subclass of $\mathcal{Y}$,
then locally $\mathcal{X}$-domains are locally $\mathcal{Y}$-domains.
Also, if $\mathcal{X}$ is stable under the quotient extension,
then the class of $\mathcal{X}$-domains is a subclass of locally $\mathcal{X}$-domains.

The following result establishes the transfer of the ($t$-)locally $\mathcal{X}$ property
to ($t$-)flat overrings.
Moreover, we show that the definition of the locally $\mathcal{X}$-domains
(respectively, the $t$-locally $\mathcal{X}$-domains)
may be reduced to localizations at a maximal ideal
(respectively, maximal $t$-ideal)
when the property $\mathcal{X}$ is stable under the quotient extension.

\begin{proposition}\label{ffalt}
Let $D$ be an integral domain and
let $\mathcal{X}$ be a property of integral domains.
Then the following assertions hold.
\begin{enumerate}
\item[(1)]
$D$ is a locally $\mathcal{X}$-domain if and only if
every flat overring of $D$ is also a locally $\mathcal{X}$-domain.
\item[(2)]
$D$ is a $t$-locally $\mathcal{X}$-domain if and only if
every $t$-flat overring of $D$ is also a $t$-locally $\mathcal{X}$-domain.
\item[(3)]
If $\mathcal{X}$ is stable under the quotient extension, then
\begin{enumerate}
\item[(a)]
$D$ is a locally $\mathcal{X}$-domain if and only if
$D_M$ has the property $\mathcal{X}$ for all $M \in {\rm Max}(D)$.
\item[(b)]
$D$ is a $t$-locally $\mathcal{X}$-domain if and only if
$D_M$ has the property $\mathcal{X}$ for all $M \in t\text{-}{\rm Max}(D)$.
\end{enumerate} 
\end{enumerate}
\end{proposition}

\begin{proof}
(1) Suppose that $D$ is a locally $\mathcal{X}$-domain.
Let $T$ be a flat overring of $D$ and
let $\mathfrak{p}$ be a prime ideal of $T$.
Then $T_{\mathfrak{p}} = D_{\mathfrak{p} \cap D}$ has the property $\mathcal{X}$.
Thus $T$ is locally $\mathcal{X}$.
The converse is obvious.

(2) The proof is similar to the proof of
the assertion (1) with the fact that
every contraction of a prime $t$-ideal of a $t$-flat overring of $D$
is a prime $t$-ideal of $D$.

(3) For the assertion (a),
suppose that $D_M$ has the property $\mathcal{X}$
for any $M \in {\rm Max}(D)$.
Let $P$ be a prime ideal of $D$.
Then there exists a maximal ideal $M$ of $D$ containing $P$.
As $D_P = (D_M)_{PD_M}$ and
the property $\mathcal{X}$ is stable under the quotient extension,
$D_P$ has the property $\mathcal{X}$.
This implies that $D$ is a locally $\mathcal{X}$-domain.
The converse is obvious.
Also, the proof of the assertion (b) is similar to that of (a)
by using the fact that every prime $t$-ideal of $D$ is contained in a maximal $t$-ideal.
\end{proof}

From the above proposition, we deduce the following two corollaries.

\begin{corollary}\label{cffalt}
Let $D$ be an integral domain.
If $\mathcal{X}$ is a property of integral domains,
then the following assertions hold.
\begin{itemize}
\item[(1)]
$D$ is a locally $\mathcal{X}$-domain if and only if
$D_S$ is a locally $\mathcal{X}$-domain
for any multiplicative subset $S$ of $D$.
\item[(2)]
If $D$ is a $t$-locally $\mathcal{X}$-domain,
then so is $D_S$ for any multiplicative subset $S$ of $D$.
\end{itemize}
\end{corollary}

\begin{proof}
Suppose that $D_S$ is a locally $\mathcal{X}$-domain
for any multiplicative subset $S$ of $D$.
Let $P$ be a prime ideal of $D$ and
set $S = D \setminus P$.
Then $D_P = (D_S)_{PD_S}$ has the property $\mathcal{X}$,
which means that $D$ is a locally $\mathcal{X}$-domain.
The converse and the assertion (2) are obvious since every quotient ring is a ($t$-)flat overring.
\end{proof}

Recall from \cite{P76} that a proper overring $T$ of
an integral domain $D$ is said to be \textit{minimal}
if there is no proper overring of $D$ which is properly contained in $T$.

\begin{corollary}
Let $D$ be an integral domain which has a minimal overring $T$ and
let $\mathcal{X}$ be a property of integral domains.
Suppose that $D$ is a locally $\mathcal{X}$-domain.
Then the following assertions hold.
\begin{enumerate}
\item[(1)]
If $D$ is not a quasi-local domain and
$T$ is a quasi-local domain, then $T$ has the property $\mathcal{X}$.
\item[(2)]
If $D$ is an integrally closed quasi-local domain,
then $T$ has the property $\mathcal{X}$.
\end{enumerate}
\end{corollary}

\begin{proof}
Note that every quasi-local minimal overring of non-quasi-local domains and
minimal overrings of integrally closed quasi-local domains are
quasi-local flat overrings
\cite[Lemma 3.1]{P76} and \cite[Lemma 2.7]{P76}, respectively.
By Proposition \ref{ffalt},
the assertions hold.
\end{proof}

Recall that an integral domain $D$ is said to be a {\it semi-quasi-local ring}
if $D$ has only finitely many maximal ideals.

\begin{corollary}\label{cffalt2}
Let $D$ be an integral domain with quotient field $K$,
$\{D_i \,|\, 1 \leq i \leq n\}$ a finite set of semi-quasi-local flat overrings of $D$ and
$T = \bigcap_{i=1}^{n} D_i$.
Suppose that $\mathcal{X}$ is a property of integral domains
which is stable under the quotient extension.
If $D$ is a locally $\mathcal{X}$-domain, then so is $T$.
\end{corollary}

\begin{proof}
Note that every nonunit of $T$ is also a nonunit in $D_i$
for some $1 \leq i \leq n$, so
the set of nonunits of $T$ is exactly the union of
the finite set of contracted maximal ideals of $D_i$.
Since for each $1 \leq i \leq n$,
the number of maximal ideals of $D_i$ is finite,
every maximal ideal of $T$ is the contraction of a maximal ideal of some $D_i$
by the prime avoidance lemma.
Now, assume that $D$ is a locally $\mathcal{X}$-domain and
let $\mathfrak{m}$ be a maximal ideal of $T$.
Then there exists $1 \leq i \leq n$ such that
$\mathfrak{m} = M_i \cap T$, where $M_i \in {\rm Max}(D_i)$.
Since $D_i$ is a flat overring of $D$, it follows that
$T_{\mathfrak{m}} = T_{M_i \cap T} = (D_i)_{M_i}$ has the property $\mathcal{X}$,
as $D_i$ is a locally $\mathcal{X}$-domain by Proposition \ref{ffalt}(1).
Hence we show that $T_{\mathfrak{m}}$ has the property $\mathcal{X}$
for any $\mathfrak{m} \in {\rm Max}(T)$.
Thus $T$ is a locally $\mathcal{X}$-domain by Proposition \ref{ffalt}(3).
\end{proof}

We next prove that the property of being a locally $\mathcal{X}$-domain is
preserved under Nagata ideal transforms
when the property $\mathcal{X}$ is stable under the quotient extension and
the considered domain is not a quasi-local domain.
First, for an ideal $I$ of an integral domain $D$ (with quotient field $K$),
we recall that the \textit{Nagata (ideal) transform of} $I$ is
the overring of $D$ defined as follows:
\begin{center}
$T(I)=\bigcup_{n=0}^\infty(D:I^n)=\bigcup_{n=0}^\infty\{a\in K \,|\, aI^n\subseteq D\}$.
\end{center}
For the sake of simplicity,
if $I = (a)$, then we denote $T((a))$ by $T(a)$.
It is worth noting that for each nonunit element $a$ of $D$,
we have: $T(a)=D_S$, where $S=\{a^n \,|\, n \geq 1\}$ (cf. \cite[Lemma 2.2]{Br68}).
We also note that $D$ is either a Pr\"ufer domain or an almost Dedekind domain if and only if
so is $T(a)$ for each nonunit element $a$ of $D$ (cf. \cite[Propositions 2.4 and 2.5]{Br68}).
Thus in the following, we show that the same result runs with any property of
integral domains that is stable under the quotient extension.

\begin{proposition}\label{NT}
Let $D$ be an integral domain which is not a quasi-local domain and
let $\mathcal{X}$ be a property of integral domains.
Then $D$ is a locally $\mathcal{X}$-domain if and only if
$T(a)$ is a locally $\mathcal{X}$-domain for any nonunit element $a \in D$.
\end{proposition}

\begin{proof}
Suppose that $D$ is a locally $\mathcal{X}$-domain.
Let $a$ be a nonunit element of $D$.
Then $T(a) = D_S$,
where $S = \{a^n \,|\, n \in \mathbb{N}\}$.
Hence by Corollary \ref{cffalt},
$T(a)$ is a locally $\mathcal{X}$-domain.
For the converse,
suppose that $T(a)$ is a locally $\mathcal{X}$-domain
for each nonunit element $a$ of $D$.
Let $P$ be a prime ideal of $D$.
As $D$ is not a quasi-local domain,
there is a nonunit element in $D$ which is not contained in $P$.
Set $S = \{a^n \,|\, n \in \mathbb{N}\}$.
Then $T(a) = D_S$.
Since $P$ is disjoint from $S$, it follows that
$PD_S$ is a prime ideal of $D_S$.
Hence $D_P = (D_S)_{PD_S}$ has the property $\mathcal{X}$,
which means that $D$ is a locally $\mathcal{X}$-domain.
\end{proof}

Let $D$ be an integral domain and
let $I$ be a nonzero finitely generated ideal of $D$.
Then $T(I)=\bigcap_{i=1}^nT(a_i)$,
where $a_i$'s are generators of $I$
\cite[Proposition 1.4]{Br68}.
Hence the next result follows directly from Corollary \ref{cffalt2} and Proposition \ref{NT}.

\begin{corollary}
Let $D$ be a semi-quasi-local domain which is not quasi-local,
$I$ a nonzero finitely generated ideal of $D$ and
$\mathcal{X}$ a property of integral domains.
If $D$ is a locally $\mathcal{X}$-domain, then so is $T(I)$.
\end{corollary}

For a given property $\mathcal{X}$ of integral domains,
an integral domain $D$ is said to be a \textit{maximal locally $\mathcal{X}$-domain} if
$D$ is not locally $\mathcal{X}$ but each proper overring of $D$ is locally $\mathcal{X}$.
We first give a necessary condition for a domain to be maximal locally $\mathcal{X}$.

\begin{proposition}
Let $D$ be an integral domain and
let $\mathcal{X}$ be a property of integral domains.
If $D$ is a maximal locally $\mathcal{X}$-domain,
then it is a quasi-local domain.
\end{proposition}

\begin{proof}
Suppose that $D$ is a maximal locally $\mathcal{X}$-domain.
Then $D$ is not a locally $\mathcal{X}$-domain,
so there is a prime ideal $P$ of $D$ such that
$D_P$ does not have the property $\mathcal{X}$.
If $D_P$ is a proper overring of $D$,
then $D_P$ is a locally $\mathcal{X}$-domain
since $D$ is a maximal locally $\mathcal{X}$-domain.
Hence $D_P$ has the property $\mathcal{X}$ as it is a quasi-local domain,
a contradiction.
Thus $D_P=D$,
which means that $D$ is a quasi-local domain.
\end{proof}

An integral domain $D$ is a {\it Mori domain} if
$D$ satisfies ascending chain condition on $v$-ideals of $D$.
Note that every valuation Mori domain is a DVR, and hence
we can easily obtain the fact that
any proper valuation overring of a maximal locally Mori domain is a DVR.

\begin{proposition}
Let $D$ be an integrally closed domain.
Then the following assertions hold.
\begin{enumerate}
\item[(1)]
If $D$ is a maximal locally SM domain,
then every proper quotient extension of $D$ is an almost Krull domain.
\item[(2)]
If $D$ is a maximal locally Mori domain,
then $D$ is either an intersection of DVRs or a valuation domain.
\end{enumerate}
\end{proposition}

\begin{proof}
(1) Let $D_S$ be a proper quotient extension of $D$,
where $S$ is a multiplicative subset of $D$.
As the integrally closed property is stable under the quotient extension,
$D_S$ is an integrally closed locally SM domain.
Since every integrally closed SM domain is a Krull domain, 
$D_S$ is an almost Krull domain.

(2) Since $D$ is integrally closed,
we have
$D=\bigcap_{ \alpha \in \Lambda} V_{\alpha}$,
where $V_\alpha$ is a valuation domain
\cite[Theorem 57]{K74}.
If there is $\alpha \in \Lambda$ such that
$D=V_\alpha$, then we are done.
Also, if $D\subsetneq V_\alpha$ for each $\alpha \in \Lambda$,
then each $V_{\alpha}$ is a DVR.
Thus the assertion holds.
\end{proof}

Let $D$ be an integral domain and
let $S$ be a multiplicative subset of $D$.
Set $N(S) = \{a \in D \setminus \{0\} \,|\,
(a,s)_v = D \text{ for all $s \in S$}\}$.
Then $N(S)$ is a saturated multiplicative subset of $D$.
In this case, $N(S)$ is called the {\it $m$-complement} of $S$.

Let $S$ be a saturated multiplicative subset of $D$.
Recall that $S$ is a {\it splitting set} if
for each $a \in D \setminus \{0\}$,
there exist $s \in S$ and $t \in N(S)$ such that
$a = st$.
Now, we state several definitions of generalizations of splitting set used in the next result.
\begin{itemize}
\item A saturated multiplicative subset $S$ is an {\it almost splitting set} if
for each $a \in D \setminus \{0\}$,
there exist a positive integer $n \in \mathbb{N}$,
$s \in S$ and $t \in N(S)$ such that
$a^n = st$.
\item A multiplicative subset $S$ is a {\it $t$-splitting set} if
for each $a \in D \setminus \{0\}$,
$(a) = (IJ)_t$ for some integral ideals $I$ and $J$ of $D$,
where $I_t \cap (s) = sI_t$ for all $s \in S$, and $J_t \cap S \neq \emptyset$.
\end{itemize}
A useful fact about the above definitions is that every almost splitting set is a $t$-splitting set \cite[Proposition 2.3]{C05}.

\begin{proposition}
Let $D$ be an integral domain and
let $S$ be an almost splitting set of $D$.
Suppose that $\mathcal{X}$ is a property of integral domains.
Then $D$ is a $t$-locally $\mathcal{X}$-domain if and only if
$D_S$ and $D_{N(S)}$ are $t$-locally $\mathcal{X}$-domains.
\end{proposition}

\begin{proof}
Suppose that $D_S$ and $D_{N(S)}$ are $t$-locally $\mathcal{X}$-domains and
let $P$ be a prime $t$-ideal of $D$.
Then $PD_S$ is a prime $t$-ideal of $D_S$
when $P$ is disjoint from $S$
\cite[Theorem 4.9]{AAZ01}.
Hence $D_P = (D_S)_{PD_S}$ has the property $\mathcal{X}$.
Now, assume that $P \cap S \neq \emptyset$.
Then it is easy to check that $P \cap N(S) = \emptyset$.
Note that $N(S)$ is an almost splitting set of $D$
\cite[Proposition 2.4]{ADZ04}.
Hence again by \cite[Theorem 4.9]{AAZ01},
$PD_{N(S)}$ is a prime $t$-ideal of $D_{N(S)}$.
This implies that $D_P = (D_{N(S)})_{PD_{N(S)}}$ has the property $\mathcal{X}$.
Thus $D_P$ has the property $\mathcal{X}$
for any $P \in t \mhyphen {\rm Spec}(D)$,
which says that $D$ is a $t$-locally $\mathcal{X}$-domain.
The converse follows directly from Corollary \ref{cffalt}(2).
\end{proof}

Now, we study some local properties of a polynomial ring and some of its quotient extensions.
Let $D$ be an integral domain and
consider the following three multiplicative subsets of $D[X]$:
\begin{eqnarray*}
A &=& \{f \in D[X] \,|\, f(0) = 1\},\\
U &=& \{f \in D[X] \,|\, f \text{ is monic}\},\text{ and}\\
N &=& \{f \in D[X] \,|\, c(f) = D\},
\end{eqnarray*}
where $c(f)$ is the ideal of $D$ generated by
the coefficients of $f$.
In this case,
$D[X]_A$ (respectively, $D[X]_U$ and $D[X]_N$)
is called the {\it Anderson ring}
(respectively, {\it Serre's conjecture ring} and
{\it Nagata ring}) of $D$.
It is clear that $D[X]_A \subsetneq D[X]_U \subseteq D[X]_N$ and
$D[X]_U$ is the quotient ring of $D[X]_A$ by $U$, and
$D[X]_N$ is the quotient ring of $D[X]_U$ by $N$.
Useful facts when studying the local properties of the above rings are that,
for $P \in {\rm Spec}(D)$,
$(D[X]_A)_{(P+XD[X])_A} = D_P[X]_{A_P}$ and $(D[X]_N)_{PD[X]_N} = D_P[X]_{N_P}$,
where $A_P = \{f \in D_P[X] \,|\, f(0) \text{ is a unit in }D_P\}$ and
$N_P = \{f \in D_P[X] \,|\, c(f) = D_P\}$
\cite[Lemma 2.6(4)]{Baek 2022} and \cite[Example 5.5.2(2)]{wang book}, respectively.

\begin{theorem}\label{local property QE}
Let $D$ be an integral domain and
let $\mathcal{X}$ be a property of integral domains
which is stable under the quotient extension and
the polynomial extension.
Suppose that any integral domain has the property $\mathcal{X}$ when
its Nagata ring has the property $\mathcal{X}$.
Then the following assertions are equivalent.
\begin{itemize}
\item[(1)]
$D$ is a locally $\mathcal{X}$-domain.
\item[(2)]
$D[X]$ is a locally $\mathcal{X}$-domain.
\item[(3)]
$D[X]_A$ is a locally $\mathcal{X}$-domain.
\item[(4)]
$D[X]_U$ is a locally $\mathcal{X}$-domain.
\item[(5)]
$D[X]_N$ is a locally $\mathcal{X}$-domain.
\end{itemize}
\end{theorem}

\begin{proof}
(1) $\Rightarrow$ (2)
Let $\mathfrak{p}$ be a prime ideal of $D[X]$ and
let $P = \mathfrak{p}\cap D$.
Then $D[X]_{\mathfrak{p}} = D_P[X]_{\mathfrak{p}D_P[X]}$ \cite[Lemma 13.1]{H88}.
Since $D_P$ has the property $\mathcal{X}$,
we obtain that $D[X]_{\mathfrak{p}}$ also has the property $\mathcal{X}$
by our assumption.
Hence $D[X]$ is a locally $\mathcal{X}$-domain.

The implications (2) $\Rightarrow$ (3) $\Rightarrow$ (4)
follow directly from Corollary \ref{cffalt}(2)
since $D[X]_A$ (respectively, $D[X]_U$ and $D[X]_N$)
is a quotient extension of $D[X]$
(respectively, $D[X]_A$ and $D[X]_U$).

(4) $\Rightarrow$ (1)
Let $M$ be a maximal ideal of $D$.
Then $MD[X]_N$ is a maximal ideal of $D[X]_N$,
so $(D[X]_N)_{MD[X]_N} = D_M[X]_{N_M}$ has the property $\mathcal{X}$.
By our assumption,
we obtain that $D_M$ has the property $\mathcal{X}$.
Thus $D$ is a locally $\mathcal{X}$-domain
by Proposition \ref{ffalt}(3).
\end{proof}

Recall that $D$ is an {\it MZ-valuation domain} if
$D_P$ is a valuation domain for any associated prime ideal $P$ of $D$.
Note that an MZ-valuation domain is exactly a locally essential domain.

The following example provides properties $\mathcal{X}$
for which integral domains satisfy the assumptions of Theorem \ref{local property QE},
as well as properties $\mathcal{X}$
for which one of the assumptions of Theorem \ref{local property QE} fails,
while the equivalent conditions of Theorem~\ref{local property QE} are still satisfied.

\begin{example}
{\rm
(1) Let $\mathcal{Y}$ be a property of an integral domain belonging to one of the following classes:
P$v$MDs, Noetherian domains, SM domains, and Krull domains.
In \cite{OT21}, the authors proved that $\mathcal{Y}$ satisfies the condition of
the property $\mathcal{X}$ in Theorem \ref{local property QE}.
Hence either $D$, $D[X]$, $D[X]_A$, $D[X]_U$, and $D[X]_N$ are locally $\mathcal{Y}$-domains or not.

(2) In \cite{H81}, the authors showed that an essential domain is not stable under the quotient extension in general.
However, in \cite{KT21}, the authors proved that
either $D$, $D[X]$, $D[X]_U$, and $D[X]_N$ are MZ-valuation domains or not.
Thus either $D$, $D[X]$, $D[X]_A$, $D[X]_U$, and $D[X]_N$ are locally essential domains or not
since locally essential domain is stable under the quotient extension.
}
\end{example}

Let $D$ be an integral domain.
Consider the set $N_v = \{f \in D[X] \,|\, c(f)_v = D\}$.
Then $N_v$ is a multiplicative subset of $D[X]$,
so we obtain the ring $D[X]_{N_v}$.
In this case, $D[X]_{N_v}$ is called the
{\it $t$-Nagata ring} of $D$.
Now, we can investigate some $t$-local properties of the polynomial ring, the Anderson ring, the Serre's conjecture ring, and the ($t$-)Nagata ring.

\begin{theorem}\label{t-local property QE}
Let $D$ be an integral domain and
let $\mathcal{X}$ be a property of integral domains which is stable under the quotient extension.
Suppose that a DVR has the property $\mathcal{X}$ and
suppose that an integral domain has the property $\mathcal{X}$
if and only if its Nagata ring has the property $\mathcal{X}$.
Then the following assertions are equivalent.
\begin{itemize}
\item[(1)]
$D$ is a $t$-locally $\mathcal{X}$-domain.
\item[(2)]
$D[X]$ is a $t$-locally $\mathcal{X}$-domain.
\item[(3)]
$D[X]_A$ is a $t$-locally $\mathcal{X}$-domain.
\item[(4)]
$D[X]_U$ is a $t$-locally $\mathcal{X}$-domain.
\item[(5)]
$D[X]_N$ is a $t$-locally $\mathcal{X}$-domain.
\item[(6)]
$D[X]_{N_v}$ is a $t$-locally $\mathcal{X}$-domain.
\item[(7)]
$D[X]_{N_v}$ is a locally $\mathcal{X}$-domain.
\end{itemize}
\end{theorem}

\begin{proof}
(1) $\Rightarrow$ (2)
Suppose that $D$ is a $t$-locally $\mathcal{X}$-domain.
Let $\mathfrak{m}$ be a maximal $t$-ideal of $D[X]$.
Then either $\mathfrak{m} \cap D = (0)$ or
$\mathfrak{m} = MD[X]$ for some $M \in t\text{-Max}(D)$.
If $\mathfrak{m} \cap D = (0)$,
then $D[X]_{\mathfrak{m}}$ is a DVR,
so $D[X]_{\mathfrak{m}}$ has the property $\mathcal{X}$ by our assumption.
On the other hand,
suppose that $\mathfrak{m} = MD[X]$
for some $M \in t\text{-Max}(D)$.
Then $D_M$ has the property $\mathcal{X}$.
Hence $D[X]_{MD[X]} = D_M[X]_{N_M}$
also has the property $\mathcal{X}$.
Hence $D[X]_{\mathfrak{m}}$ has the property $\mathcal{X}$.
By Proposition \ref{ffalt},
$D[X]$ is a $t$-locally $\mathcal{X}$-domain.

The implications (2) $\Rightarrow$ (3) $\Rightarrow$ (4) $\Rightarrow$ (5) $\Rightarrow$ (6) hold
since the property $\mathcal{X}$ is stable under the quotient extension,
and the implication (6) $\Rightarrow$ (7) holds since $D[X]_{N_v}$ is a DW-domain.

(7) $\Rightarrow$ (1)
Suppose that $D[X]_{N_v}$ is a locally $\mathcal{X}$-domain.
Let $M$ be a maximal $t$-ideal of $D$.
Then $MD[X]_{N_v}$ is a maximal ideal of $D[X]_{N_v}$.
Hence $D_M[X]_{N_M} = (D[X]_{N_v})_{MD[X]_{N_v}}$ has the property $\mathcal{X}$,
so $D_M$ has the property $\mathcal{X}$
by our assumption.
Thus $D$ is a $t$-locally $\mathcal{X}$-domain.
\end{proof}

\begin{example}
{\rm
It is easy to show that every DVR is a Krull domain,
and $D$ is a Krull domain if and only if $D[X]_N$ is a Krull domain \cite[Theorem 5.2(1)]{AAM85}.
Hence either $D$, $D[X]$, $D[X]_A$, $D[X]_U$, $D[X]_N$, and $D[X]_{N_v}$
are $t$-locally Krull domains or not.
Moreover, $D$ is a $t$-locally Krull domain if and only if $D[X]_{N_v}$ is a locally Krull domain.
}
\end{example}

\section{Locally Krull-like domains in pullback constructions}

In this section, we will focus on the transfer of some local properties to pullbacks.

To avoid unnecessary repetition, let us fix some necessary notation.
Let $T$ be an integral domain, $M$ a maximal ideal,
$K = T/M$, and $D$ a proper subring of $K$ with the quotient field $k$.
Consider the canonical epimorphism $\varphi: T\to K$.
Then we obtain the following pullback diagram.

\begin{center}
$
\begin{tikzcd}
R:=\varphi^{-1}(D) \ar[r] \ar[d, hook] & D \ar[d, hook] \\
T \ar[r,"\varphi"] & K=T/M
\end{tikzcd}
$
\end{center}
We shall refer to this as a pullback diagram of type $(\square)$.
The reader can refer to \cite{FG96,GH00}
for more details on the ideal structure of
this construction and its ring-theoretic properties.
By localizing $M$, we can also obtain the following pullback diagram ($\square_M$).
\begin{center}
$
\begin{tikzcd}
R_M \ar[r] \ar[d, hook] & k \ar[d, hook] \\
T_M \ar[r] & K
\end{tikzcd}
$
\end{center}
Note that the pullback diagram ($\square_M$) may be considered as of type ($\square$) when $D$ is a field.

Before we start,
we investigate some useful facts in the above pullback diagram of type ($\square$).
The statements of the following lemma may be found in
\cite{ABDFK88} and \cite[Section 1]{GH00}
which will be used frequently in the sequel without explicit mention.

\begin{lemma}\label{LemPB1}
For the pullback diagram of type $(\square)$,
the following assertions hold.
\begin{enumerate}
\item[(1)]
$M=(R:T)$ and $D\cong R/M$.
\item[(2)]
If $P$ is a prime ideal $($respectively, maximal ideal$)$ of $D$,
then $\varphi^{-1}(P)$ is a prime ideal $($respectively, maximal ideal$)$ of $R$.
\item[(3)]
If $P$ is a prime ideal $($respectively, maximal ideal$)$ of $R$ with $M\subseteq P$,
then there is a unique prime ideal $($respectively, maximal ideal$)$ $Q$ of $D$ such that $P=\varphi^{-1}(Q)$.
\item[(4)]
If $P$ is a prime ideal $($respectively, maximal ideal$)$ of $R$ which does not contain $M$,
then there is a unique prime ideal $($respectively, maximal ideal$)$ $Q$ of $T$ such that
$P=Q\cap R$.
In this case, $R_P=T_Q$.
\item[(5)]
$k=K$ if and only if $R_M=T_M$, if and only if $T$ is a flat overring of $R$.
\end{enumerate}
\end{lemma}

In \cite{KOT22}, the authors proved that
for the pullback diagram of type ($\square$),
$R$ is a locally Noetherian domain if and only if
$T$ is a locally Noetherian domain, $D$ is a field and $[K:k]$ is finite.
Motivated by this result,
it is natural to ask whether an analogous statement holds
in the setting of SM domains.
In order to address this question,
we first recall the following preparatory lemmas.

\begin{lemma}[{\cite[Theorem 2.16]{OT21}}]\label{LemPB2}
For the pullback diagram of type $(\square)$,
$R$ is a $t$-locally Mori domain if and only if
$D$ is a field,
$T$ is a $t$-locally Mori domain
and $T_M$ is a Mori domain.
\end{lemma}

\begin{lemma}[{\cite[Theorem 3.11]{M03}}]\label{LemPB3}
For the pullback diagram of type $(\square)$,
$R$ is an SM domain if and only if
$D$ is a field,
$T$ is an SM domain,
$T_M$ is Noetherian
and $[K:k]$ is finite.
\end{lemma}

Now, we are ready to study the transfer of the locally SM notion to pullbacks.

\begin{theorem}\label{Pullback SM}
For the pullback diagram of type $(\square)$,
$R$ is a locally SM domain if and only if
$D$ is a field,
$T$ is a locally SM domain,
$T_M$ is a Noetherian domain
and $[K:k]$ is finite.
\end{theorem}

\begin{proof}
Suppose that $R$ is a locally SM domain.
Then $R$ is a $t$-locally Mori domain,
so $D$ is a field, and hence $M$ is a maximal ideal of $R$
by Lemma \ref{LemPB2}.
Hence the diagram ($\square_M$) may be considered as of type ($\square$).
By Lemma \ref{LemPB3},
$T_M$ is a Noetherian domain and $[K:k]$ is finite
since $R_M$ is an SM domain.
Now, we claim that $T$ is a locally SM domain.
If $T$ is a quasi-local domain,
then $T = T_M$ is a Noetherian domain,
and hence $T$ is a locally SM domain.
Now, assume that $T$ is not a quasi-local domain and
let $N$ be a maximal ideal of $T$ distinct from $M$.
Then $P := N \cap R$ is a prime ideal of $R$
which does not contain $M$.
Hence $T_N = R_P$ is an SM domain.
Since the SM-property is stable under the quotient extension,
$T$ is a locally SM domain by Proposition \ref{ffalt}(3).
For the converse,
suppose that the latter condition holds and
let $N$ be a maximal ideal of $R$.
If $N$ is distinct from $M$,
then there exists a maximal ideal $\mathfrak{m}$ of $T$ such that
$N = \mathfrak{m} \cap R$,
and hence $R_N = T_{\mathfrak{m}}$ is an SM domain.
On the other hand, if $N = M$,
then the diagram ($\square_M$) may be considered as of type ($\square$).
This implies that $R_N$ is an SM domain by Lemma \ref{LemPB3}.
Since the SM-property is stable under the quotient extension,
$R$ is a locally SM domain by Proposition \ref{ffalt}(3).
\end{proof}

It is worth pointing out that the argument above yields more than just the proof of the previous theorem;
indeed, we have the following remark.

\begin{remark}
{\rm
Consider the pullback diagram of type $(\square)$.
In the proof of Theorem \ref{Pullback SM},
we showed that if $R$ is a locally SM domain,
then $T$ is a Noetherian domain, $D$ is a field and $[K:k]$ is finite
when $T$ is a quasi-local domain.
Hence $R$ is a Noetherian domain \cite[Theorem 4.6]{GH00}.
This implies that the notions of a locally SM domain and a Noetherian domain
(and hence an SM domain) coincide when $T$ is a quasi-local domain.
Thus if we want to construct a locally SM domain which is not an SM domain
issued from pullbacks,
it is necessary to start with the condition `$T$ is not a quasi-local domain'.
}
\end{remark}

Also, since a Krull domain is exactly both an SM domain and a P$v$MD,
we obtain the following remark.

\begin{remark}
{\rm
Consider the pullback diagram of type $(\square)$.
Suppose that $R$ is either a Krull domain or locally Krull domain.
Then $R$ is both a locally SM domain and a locally P$v$MD.
This implies that $k = K$ by Theorem \ref{Pullback SM} and \cite[Corollary 3.3]{KN97}.
This contradicts the fact that $D$ is a proper subring of $K$.
Thus $R$ can never be a Krull domain nor a locally Krull domain.
}
\end{remark}

The fact that a (locally) Krull domain can never occur as
a pullback naturally leads to the question of whether
a pullback can be one of some generalizations of Krull domains.
Thus we proceed to investigate when a pullback can be one of some generalizations of Krull domains.
For our goal, we first establish the following lemma.

\begin{lemma}\label{Krull-like domain}
For the pullback diagram of type $(\square)$,
the following assertions hold.
\begin{itemize}
\item[(1)]
$R$ is a weakly Krull domain if and only if
$D$ is a field,
$T$ is a weakly Krull domain
and ${\rm ht}(M) = 1$.
\item[(2)]
$R$ is an infra-Krull domain if and only if
$D$ is a field,
$T$ is an infra-Krull domain,
${\rm ht}(M) = 1$
and $[K:k]$ is finite.
\item[(3)]
$R$ is a Krull-type domain if and only if
$D$ is a semi-quasi-local B\'ezout domain,
$T$ is a Krull-type domain,
$T_M$ is a valuation domain
and $k = K$.
\end{itemize}
\end{lemma}

\begin{proof}
The assertions (1) and (3) appear in \cite[Theorem 1]{C04} and \cite[Proposition 4.1]{KM07}, respectively.

(2) Suppose that $R$ is an infra-Krull domain.
Then $R$ is both a strong Mori domain and a weakly Krull domain.
By the assertion (1) and Lemma \ref{LemPB3},
$D$ is a field, ${\rm ht}(M) = 1$ and $[K:k]$ is finite.
Now, we claim that $T$ is an infra-Krull domain.
It is sufficient to show that $T_Q$ is a Noetherian domain for any $Q \in X^1(T)$
since $T$ is a weakly Krull domain.
As $D$ is a field,
the pullback diagram $(\square_M)$ may be considered as of type $(\square)$.
Now, note that $R_M$ is a Noetherian domain,
so $T_M$ is a Noetherian domain by Lemma \ref{LemPB3}.
Let $Q \in X^1(T)$ with $Q \neq M$ and $P = Q \cap R$.
Then $P \in X^1(R)$ and
$T_Q = R_P$ \cite[Corollary 2]{C04}.
Hence $T_Q = R_P$ is a Noetherian domain.
Therefore $T$ is an infra-Krull domain.
For the converse,
suppose that the latter condition holds.
Then $R$ is a weakly Krull domain by the assertion (1).
Let $P \in X^1(R)$.
Then there exists $Q \in X^1(T)$ such that $P = Q \cap R$ \cite[Corollary 2]{C04}.
Since $T$ is an infra-Krull domain,
$T_Q$ is a Noetherian domain.
If $Q \neq M$, then $R_P = T_Q$ is a Noetherian domain,
where the equality follows from \cite[Corollary 2]{C04}.
On the other hand, let $Q = M$.
Then $P = M$.
By localizing at $M$,
we can obtain the pullback diagram $(\square_M)$ of type $(\square)$.
Hence $R_P = R_M$ is a Noetherian domain \cite[Theorem 4.6]{GH00}.
Thus $R$ is an infra-Krull domain.
\end{proof}

The next result represents a partial step toward our goal of
determining when a pullback can be a Krull-like domain.

\begin{theorem}\label{locally Krull-like domain}
For the pullback diagram of type $(\square)$,
the following assertions hold.
\begin{itemize}
\item[(1)]
$R$ is a locally weakly Krull domain if and only if
$D$ is a field,
$T$ is a locally weakly Krull domain
and ${\rm ht}(M) = 1$.
\item[(2)]
$R$ is a locally infra-Krull domain if and only if
$D$ is a field,
$T$ is a locally infra-Krull domain,
${\rm ht}(M) = 1$
and $[K:k]$ is finite.
\item[(3)]
$R$ is a locally Krull-type domain if and only if
$D$ is a locally B\'ezout domain,
$T$ is a locally Krull-type domain,
$T_M$ is a valuation domain
and $k = K$.
\end{itemize}
\end{theorem}

\begin{proof}
(1) Suppose that $R$ is a locally weakly Krull domain.
Let $\mathfrak{m}$ be a maximal ideal of $D$ and
let $\mathfrak{p} := \varphi^{-1}(\mathfrak{m})$.
Then $\mathfrak{p}$ is a maximal ideal of $R$ containing $M$.
Hence we can obtain the following pullback diagram of type $(\square)$.
\begin{center}
$
\begin{tikzcd}
R_{\mathfrak{p}} \ar[r] \ar[d, hook] & D_{\mathfrak{m}} \ar[d, hook] \\
T_M \ar[r] & K
\end{tikzcd}
$
\end{center}
Since $R_{\mathfrak{p}}$ is a weakly Krull domain,
$D_{\mathfrak{m}}$ is a field, ${\rm ht}(MT_M) = 1$ and $T_M$ is a weakly Krull domain,
so $D$ is a field, ${\rm ht}(M) = 1$
by Lemma \ref{Krull-like domain}(1).
Now, we claim that $T$ is a locally weakly Krull domain.
Let $Q$ be a maximal ideal of $T$ and
let $P := Q \cap R$.
If $Q \neq M$,
then $P$ is a prime ideal of $R$ which does not contain $M$.
Hence $T_Q = R_P$ is a weakly Krull domain.
Therefore $T$ is a locally weakly Krull domain
by Proposition \ref{ffalt}(3).
For the converse, suppose that the latter condition holds.
Let $P$ be a maximal ideal of $R$.
If $P$ does not contain $M$,
then there exists a maximal ideal $Q$ of $T$ such that
$P = Q\cap R$ and $R_P = T_Q$.
Hence $R_P$ is a weakly Krull domain.
On the other hand,
if $P = M$,
then we have the pullback diagram $(\square_M)$
of type $(\square)$.
Since $T_M$ is a weakly Krull domain,
${\rm ht}(MT_M) = {\rm ht}(M) = 1$ and
$k$ is a field,
$R_M$ is a weakly Krull domain by Lemma \ref{Krull-like domain}(1).
Thus $R$ is a locally weakly Krull domain.

(2) Suppose that $R$ is a locally infra-Krull domain.
Then $R$ is a locally weakly Krull domain,
so ${\rm ht}(M) = 1$ and $D$ is a field
by the assertion (1).
As $D$ is a field,
the pullback diagram $(\square_M)$ is of type $(\square)$.
Hence $T_M$ is an infra-Krull domain and
$[K:k]$ is finite by Lemma \ref{Krull-like domain}(2).
Let $\mathfrak{m}$ be a maximal ideal of $T$
with $\mathfrak{m} \neq M$ and
$P = \mathfrak{m} \cap R$.
Then $T_{\mathfrak{m}} = R_P$ is an infra-Krull domain.
Hence $T$ is a locally infra-Krull domain.
For the converse,
suppose that the latter condition holds.
Let $P$ be a maximal ideal of $R$.
If $P \neq M$,
then there exists a maximal ideal $Q$ of $T$ such that
$P = Q \cap R$ and $R_P = T_Q$.
Hence $R_P$ is an infra-Krull domain.
On the other hand,
if $P = M$,
then the pullback diagram $(\square_M)$ is of type $(\square)$ since $D$ is a field.
Since $T_M$ is an infra-Krull domain,
${\rm ht}(MT_M) = {\rm ht}(M) = 1$,
$k$ is a field and $[K:k]$ is finite,
$R_M$ is an infra-Krull domain.
Thus $R$ is a locally infra-Krull domain.

(3) Suppose that $R$ is a locally Krull-type domain.
Let $Q$ be a nonzero prime ideal of $D$ and
$P = \varphi^{-1}(Q)$.
Then $P$ is a prime ideal of $R$ containing $M$.
Hence we obtain the following pullback diagram of type $(\square)$.
\begin{center}
$
\begin{tikzcd}
R_P \ar[r] \ar[d, hook] & D_Q \ar[d, hook] \\
T_M \ar[r] & K
\end{tikzcd}
$
\end{center}
Since $R_P$ is a Krull-type domain, $k = K$,
$D_Q$ is a B\'ezout domain and $T_M$ is a valuation domain
by Lemma \ref{Krull-like domain}(3).
This implies that $T$ is a flat overring of $R$,
and hence $T$ is a locally Krull-type domain by Proposition \ref{ffalt}(1).
For the converse,
suppose that the latter condition holds.
Let $P$ be a prime ideal of $R$.
If $P = M$, then $R_P = R_M = T_M$ is a Krull-type domain
since $k = K$.
If $P$ contains $M$,
then there exists a prime ideal $Q$ such that $\varphi^{-1}(Q) = P$,
and hence we can obtain the following pullback diagram of $(\square)$.
\begin{center}
$
\begin{tikzcd}
R_P \ar[r] \ar[d, hook] & D_Q \ar[d, hook] \\
T_M \ar[r] & K
\end{tikzcd}
$
\end{center}
Since the quotient field of $D_Q$ is equal to $K$,
$T_M$ is both a valuation domain and a Krull-type domain, and $D_Q$ is a quasi-local B\'ezout domain,
$R_P$ is a Krull-type domain by Lemma \ref{Krull-like domain}(3).
On the other hand,
if $P$ does not contain $M$,
then there exists a prime ideal $Q$ of $T$ such that
$Q \cap R = P$ and $R_P = T_Q$.
Hence $R_P$ is a Krull-type domain.
Thus $R_P$ is a Krull-type domain for any $P \in {\rm Spec}(R)$.
Consequently, $R$ is a locally Krull-type domain.
\end{proof}

Recall that
generalized Krull domains (respectively, locally generalized Krull domains)
are Krull-type domains (respectively, locally Krull-type domains).
The next remark gives a nice tool to construct a Krull-type domain
(respectively, locally Krull-type domain)
which is not a generalized Krull domain (respectively, locally generalized Krull domain).

\begin{remark}
{\rm
Consider the pullback diagram of type $(\square)$.
Suppose that $R$ is either a generalized Krull domain or
a locally generalized Krull domain.
Then $R$ is a locally weakly Krull domain,
so $D$ is a field by Theorem \ref{locally Krull-like domain}.
In contrast,
since $R$ is a locally P$v$MD,
$D = K$ \cite[Corollary 3.3]{KN97}.
This contradicts the fact that
$D$ is a proper subring of $K$.
Thus $R$ can never be a generalized Krull domain nor a locally generalized Krull domain.
A concrete example illustrating this impossibility is as follows:
Let $D \subseteq E$ be an extension of integral domains,
we have $R:=D+XE[X]$ is a generalized Krull domain if and only if $D = E$ is a generalized Krull domain
(cf. \cite[Corollary 1.7(1)]{Lim 2012}).
This says that if $D+XE[X]$ is a generalized Krull domain,
then $D+XE[X]$ is never issued from thee pullback diagram of type $(\square)$.

On the other hand,
we can construct a Krull-type domain and a locally Krull-type domain
using Lemma \ref{Krull-like domain}(3) and Theorem \ref{locally Krull-like domain}(3).
} 
\end{remark}

In \cite{EB02}, El Baghdadi defined a new type of Krull-like domain,
which differs from the four previously known Krull-like domains
({\it i.e.}, weakly Krull domains, infra-Krull domains, generalized Krull domains, and Krull-type domains).
An integral domain $D$ is said to be 
a {\it generalized Krull domain} if
it is a strongly discrete P$v$MD 
({\it i.e.}, $D_M$ is a strongly discrete valuation domain for every maximal $t$-ideal $M$ of $D$)
and every principal ideal has only finitely many minimal prime ideals.
Equivalently, $D$ is a P$v$MD and
for each prime $t$-ideal $P$ of $D$,
there exists a finitely generated ideal $J$ of $D$ such that $P = \sqrt{J_t}$ and $P \neq (P^2)_t$.
For the sake of avoiding confusion, we will use the abbreviation a {\it GK-domain}
to refer to generalized Krull domains in the sense of El Baghdadi.
It is important to point out that the definitions of (classical)
generalized Krull domains (in the sense of Gilmer) and GK-domains do not imply each other.
Note that the class of GK-domains falls strictly between
Krull domains and P$v$MDs, and that GK-domains are stable under the quotient extension.

Recall that $R$ is a GK-domain if and only if
$D$ and $T$ are GK-domains, $K=k$ and $T_M$ is a valuation domain \cite[Theorem 4.3]{EB02}.

\begin{theorem}\label{GK domain}
For a pullback diagram of type $(\square)$,
$R$ is a locally GK-domain if and only if
$D$ and $T$ are locally GK-domains,
$T_M$ is a valuation domain
and $k=K$.
\end{theorem}

\begin{proof}
Assume that $R$ is a locally GK-domain.
Then $R$ is a locally P$v$MD,
so $K=k$, $T_M$ is
a valuation domain \cite[Corollary 3.3]{KN97}.
This implies that $T$ is a flat overring of $R$,
and hence $T$ is a locally GK-domain by Proposition \ref{ffalt}(1).
Now, we claim that $D$ is a locally GK-domain.
Let $Q$ be a maximal ideal of $D$ and let $P:=\varphi^{-1}(Q)$.
Then $P$ is a prime ideal of $R$ containing $M$.
Hence we obtain the following pullback diagram of type $(\square)$.
\begin{center}
$
\begin{tikzcd}
R_P \ar[r] \ar[d, hook] & D_Q \ar[d, hook] \\
T_M \ar[r] & K
\end{tikzcd}
$
\end{center}
Thus $D_Q$ is a GK-domain by \cite[Theorem 4.3]{EB02}.
Consequently, $D$ is a locally GK-domain by Proposition \ref{ffalt}(3).
For the converse, suppose that the latter condition holds.
Let $P$ be a maximal ideal of $R$.
If $P=M$, then $R_P=T_M$ is a GK-domain.
Now, if $P \neq M$, then there is a maximal ideal $Q$ of $T$ such that
$Q\cap R=P$ and $R_P=T_Q$, and hence $R_P$ is a GK-domain.
Therefore $R$ is a locally GK-domain by Proposition \ref{ffalt}(3).
\end{proof}

By combining Lemma \ref{Krull-like domain},
Theorems \ref{locally Krull-like domain} and \ref{GK domain},
we obtain

\begin{corollary}\label{Krull in composite}
Let $K$ be a field and let $D$ be a $($proper$)$ subring of $K$ with the quotient field $k$.
Suppose that $R$ is of the form $D+XK[X]$ or $D+XK[\![X]\!]$.
Then the following assertions hold.
\begin{enumerate}
\item[(1)]
The following conditions are equivalent.
\begin{enumerate}
\item[(i)]
$R$ is a $($locally$)$ weakly Krull domain.
\item[(ii)]
$D$ is a field.
\end{enumerate}
\item[(2)]
The following conditions are equivalent.
\begin{enumerate}
\item[(i)]
$R$ is a $($locally$)$ Noetherian domain.
\item[(ii)]
$R$ is a $($locally$)$ SM domain.
\item[(iii)]
$R$ is a $($locally$)$ infra-Krull domain.
\item[(iv)]
$D$ is a field and $[K:k]$ is finite.
\end{enumerate}
\item[(3)] The following conditions are equivalent.
\begin{enumerate}
\item[(i)]
$R$ is a Krull-type domain.
\item[(ii)]
$D$ is a semi-quasi-local B\'ezout domain and $k=K$.
\end{enumerate}
\item[(4)] The following conditions are equivalent.
\begin{enumerate}
\item[(i)]
$R$ is a locally Krull-type domain.
\item[(ii)]
$D$ is a locally B\'ezout domain and $k=K$.
\end{enumerate}
\item[(5)] The following conditions are equivalent.
\begin{enumerate}
\item[(i)]
$R$ is a locally GK-domain.
\item[(ii)]
$D$ is a locally GK-domain and $k=K$.
\end{enumerate}
\end{enumerate}
\end{corollary}

At the end of this section,
we examine constructed examples of integral domains possessing one property while failing to satisfy another.
By Corollary \ref{Krull in composite}, we obtain

\begin{example}
{\rm
(1) As $[\mathbb{R}:\mathbb{Q}]=\infty$,
$\mathbb{Q}+X\mathbb{R}[X]$ and $\mathbb{Q}+X\mathbb{R}[\![X]\!]$ are
(locally) weakly Krull domains which are neither
(locally) infra-Krull domains nor Noetherian domains
by Corollary \ref{Krull in composite}(1) and (2).

(2) As $\mathbb{Z}$ is a B\'ezout domain which has infinitely many maximal ideals,
$\mathbb{Z} + X \mathbb{Q}[X]$ and $\mathbb{Z} + X \mathbb{Q}[\![X]\!]$ are
locally Krull-type domains which are not Krull-type domains
by Corollary \ref{Krull in composite}(3) and (4).
}
\end{example}


\begin{thebibliography}{99}

\bibitem{AAM85} D. D. Anderson, D. F. Anderson and R. Markanda,
{\it The rings $R(X)$ and $R\langle X\rangle$},
J. Algebra 95 (1985) 96-115.

\bibitem{AAZ01} D. D. Anderson, D. F. Anderson and M. Zafrullah, 
{\it The ring $D + XD_S[X]$ and $t$-splitting sets}, 
Commutative Algebra Arab. J. Sci. Eng. Sect. C Theme Issues 26 (1) (2001) 3-16.

\bibitem{ADZ04} D. D. Anderson, T. Dumitrescu and M. Zafrullah, 
{\it Almost splitting sets and AGCD domains}, 
Comm. Algebra 32 (2004) 147-158.

\bibitem{AMZ 1992} D. D. Anderson, J. Mott and M. Zafrullah,
{\it Finite character representations for integral domains},
Boll. Un. Mat. Ital. 6 (1992) 613–630.

\bibitem{ABDFK88} D. F. Anderson, A. Bouvier, D. E. Dobbs, M. Fontana and S. Kabbaj,
{\it On Jaffard domains},
Expo. Math. 6 (1988) 145-175.

\bibitem{Baek 2022} H. Baek and J. W. Lim,
{\it A special subring of the Nagata ring and the Serre’s conjecture ring},
arXiv: 2408.08758

\bibitem{Br68} J. Brewer,
{\it The ideal transform and overrings of an integral domain},
Math. Z. 407 (1968) 301-306.

\bibitem{C05} G. W. Chang, 
{\it Almost splitting sets in integral domains}, 
J. Pure Appl. Algebra 197 (2005) 279-292.

\bibitem{C04} G. W. Chang,
{\it Weakly Krull and related pullback domains},
Pure and Appl. Math. 11(2) (2004) 117-125.

\bibitem{EB02} S. El Baghdadi,
{\it On a class of Pr\"ufer $v$-multiplication domains},
Comm. Algebra 30 (2002) 3723-3742.

\bibitem{FG96} M. Fontana and S. Gabelli,
{\it On the class group and the local class group of a pullback},
J. Algebra 181(3) (1996) 803-835.

\bibitem{GH00} S. Gabelli and E. G. Houston,
{\it Ideal theory in pullbacks, in: Non-Noetherian Commutative Ring Theory},
Math. Appl. 520, Kluwer Academic, Dordrecht, 2000.

\bibitem{GT18} L. T. N. Giau and P. T. Toan,
{\it On generalized Krull power series rings},
Bull. Korean Math. Soc. 55 (4) (2018) 1007–1012.

\bibitem{Gi72} R. Gilmer,
{\it Multiplicative ideal theory},
Dekker, New York, 1972.

\bibitem{Gi64} R. Gilmer,
{\it Integral domains which are almost Dedekind},
Proc. Amer. Math. Soc. 15 (1964) 813-818.

\bibitem{H81} W. Heinzer,
{\it An essential integral domain with a non-essential localization},
Can. J. Math. 33 (1981) 400-403. 

\bibitem{H88} J. A. Huckaba,
{\it Commutative Rings with Zero Divisors},
Marcel Dekker, 1988.

\bibitem{KM07} S. Kabbaj and A. Mimouni,
{\it $t$-Class semigroups of integral domains},
J. Reine Angew. Math. 612 (2007) 213-229.

\bibitem{K89} B. G. Kang,
{\it Pr\"{u}fer $v$-multiplication domains and the ring $R[X]_{N_v}$},
J. Algebra 123 (1989) 151-170.

\bibitem{K74} I. Kaplansky,
{\it Commutative Rings},
The University of Chicago Press, Chicago, 1974.

\bibitem{KN97} M. Khalis and D. Nour El Abidine,
{\it On the class group of a pullback},
Lecture Notes in Pure and Applied Mathematics 185, Marcel Dekker, New York, 1997.

\bibitem{Kim 2011} H. Kim,
{\it Overrings of t-coprimely packed domains},
J. Korean Math. Soc. 48 (1) (2011) 191–205.

\bibitem{KOT22} H. Kim, O. Ouzzaouit and A. Tamoussit,
{\it Noetherian-like properties and zero-dimensionality in some extensions of rings},
Afr. Mat. 34 (2023) 42.

\bibitem{KT21} H. Kim and A. Tamoussit,
{\it Integral domains issued from associated primes},
Comm. Algebra 50(2) (2022) 538-555.

\bibitem{KP95} D.J. Kwak and Y.S. Park, 
{\it On $t$-flat overrings}, 
Chinese J. Math. 23 (1995) 17-24.

\bibitem{Lim 2012} J. W. Lim,
{\it Generalized Krull domains and the composite semigroup ring $D+E[\Gamma^*]$},
J. Algebra 357 (2012) 20-25.

\bibitem{M03} A. Mimouni,
{\it TW-domains and Strong Mori domains},
J. Pure Appl. Algebra 177 (2003) 79-93.

\bibitem{MZ81} J. L. Mott and M. Zafrullah,
{\it On Pr\"ufer $v$-multiplication domains},
Manuscripta Math. 35 (1981) 1-26.

\bibitem{OT21} O. Ouzzaouit and A. Tamoussit,
{\it On the transfer of some $t$-locally properties},
Hacettepe J. Math. Stat. 50(3) (2021) 825-832. 

\bibitem{P76} I. J. Papick,
{\it Local minimal overrings},
Canad. J. Math. 27 (1976) 788-792. 

\bibitem{P68} E. M. Pirtle,
{\it Integral domains which are almost Krull},
J. Sci. Hiroshima Univ. Ser. A-I Math. 32(2) (1968) 441-447.

\bibitem{wang book} F. G. Wang and H. Kim,
{\it Foundations of Commutative Rings and Their Modules},
Algebra and Applications, 22, Springer, Singapore, 2016.

\end{thebibliography}
\end{document}